\documentclass[12pt]{article}
\usepackage{amsmath, amssymb,tikz,cancel,graphicx,amsthm}
\usepackage{color,enumitem}
\usetikzlibrary{shapes, patterns}
\usepackage{enumitem}
\usetikzlibrary{shapes, patterns}
\usepackage[width=170mm,height=230mm,left=25mm,foot=10mm]{geometry}

\newcommand{\startClaims}{\setcounter{clm}{0}}
\newtheorem{obs}{Observation}
\newtheorem{thm}[obs]{Theorem}
\newtheorem{clm}{Claim}

\newcommand{\ignore}[1]{}

\newcommand{\apriltwothree}[1]{#1}
\newcommand{\apriltwoseven}[1]{#1}
\newcommand{\apriltwosix}[1]{#1}
\newcommand{\mayfive}[1]{#1}
\newcommand{\maysix}[1]{#1}

\begin{document}
	  \begin{center}
		{\bf{\Large Remarks on the structure of simple drawings of $K_n$}}
		\\    
		\vspace*{1cm}
		{\large R. Bruce Richter\footnote{Supported by NSERC Grant 50503-10940-500} and Matthew Sullivan}
		\\
		\vspace{0.5cm}
		{\it University of Waterloo,  Waterloo, On N2L 3G1, Canada}
		\\
		\texttt{brichter,m8sulliv@uwaterloo.ca}
		\\
		\today
		
	\end{center}
	\pagenumbering{arabic}

\begin{abstract}
\apriltwosix{In studying properties of simple drawings of the complete graph in the sphere, two natural questions arose for us:  can an edge have multiple segments on the boundary of the same face? and is each face the intersection of sides of 3-cycles?  The second is asserted to be obvious in two previously published articles, but when asked, authors of both papers were unable to provide a proof.  We present a proof.   The first is quite easily proved and the technique yields a third, even simpler, fact:  no three edges at a vertex all have internal points incident with the same face.}
\end{abstract}

\section{Introduction}

In our study of simple drawings of the complete graph in the sphere, two basic questions arose for us:  

\begin{itemize}[topsep=0pt,itemsep=-6pt]
\item Can an edge have multiple segments on the boundary of a face? 
\item Is each face the intersection of sides of 3-cycles?
\end{itemize}  

It turns out that \cite{1,2} both assume a positive answer to the second question.  When we submitted queries (including the authors of both articles), no one was able to provide a proof to the second question. 

For the first question, we had a fairly simple proof.  The one we give here incorporates ideas from both Aichholzer and Kyn\v cl (both of whom had thought about this question before).

In our drawings, each vertex of $K_n$ is a distinct point of the \apriltwosix{sphere $\mathbb S^2$ (or plane}; in this work there is no essential difference), while each edge is represented by an arc joining the points representing its incident vertices.  No edge has a vertex-point in its interior and any two edge-arcs have finitely many intersections, all of which are \apriltwosix{either common endpoints or} crossings.  We do not distinguish between a vertex or edge and the point or arc that represents it in the drawing.

We are concerned with {\em simple\/} drawings, in which any two closed edges have at most one point in common:  either a vertex or a crossing. \apriltwosix{(These are also known as ``good drawings"; there is a growing use of ``simple" as a more descriptive adjective.)}  A {\em face\/} of a drawing $D$ of $K_n$ in $\mathbb S^2$ is a component of $\mathbb S^2\setminus D[K_n]$.

One fact that we need \cite{3} is that, for $n\ge 3$,  each face of $D$ is an open disc whose closure is a closed disc.  Its boundary is a simple closed curve.

Our two theorems are the following\apriltwoseven{, proved in Sections \ref{sec:connectedIntersection} and \ref{sec:regionsTriangles}, respectively.}

\begin{thm}\label{th:connectedIntersection} Let $n\ge 3$ and let $C$ be the simple closed curve bounding a  face of a simple drawing of $K_n$ in the sphere.  If $e$ is any closed edge of $K_n$, then $e\cap C$ is either connected or consists of the two vertices incident with $e$.
\end{thm}

\begin{thm}\label{th:regionsTriangles} Let $n\ge 3$ and let $R$ be a face of a simple drawing $D$ of $K_n$ in the sphere.  For each 3-cycle $T$ of $K_n$, let $S^R_T$ denote the side of $T$ in $D$ that contains $R$.  Then $R=\bigcap S^R_T$, where the intersection is over all 3-cycles in $K_n$.
\end{thm}

\apriltwoseven{The quite elementary inductive proof of Theorem \ref{th:regionsTriangles} is the main contribution of this work.  However, given that this result seemed obvious to earlier authors, it is certainly 
natural to wonder whether there is a very direct proof.}

\mayfive{We appreciate our private exchanges with Jan Kyn\v cl for several observations and for references \cite{bfk,molnar}.  For example, he observed that our inductive proof is quite similar to the proof in Balko et al \cite{bfk} of their version of Carath\'eodory's Theorem:  if $p$ is a point of the plane in a bounded face of a simple drawing of $K_n$, then $p$ is in the interior of a 3-cycle.  They further prove that the union of the bounded faces is covered by at most $n-2$ interiors of 3-cycles.}  

\mayfive{Of related interest, Moln\'ar proved a Helly-type theorem for sets of closed discs in the plane whose boundary curves intersect finitely \cite{molnar}:  if the intersection of every two discs in the set has non-empty connected interior, then the intersection of any $k\ge 1$ discs in the set has non-empty connected interior.  \maysix{We note that} two of our 3-cycles might well have disconnected intersection of interiors.}

We conclude this section with a very simple result that illustrates the main idea.

\begin{thm}\label{th:noThreeAtVertex}  Let $n\ge 4$ and let \apriltwosix{$C$ be the simple closed curve bounding} a face of a simple drawing of $K_n$ \apriltwosix{in the sphere}.  If $e_1,e_2,e_3$ are distinct \mayfive{(open) edges incident with a common vertex $v$, then at least one of \apriltwosix{$C\cap e_1$, $C\cap e_2$, and $C\cap e_3$} is empty}.
\end{thm}

\begin{proof}  The three edges $e_1,e_2,e_3$ induce a \apriltwothree{simple drawing of} $K_4$.   \apriltwosix{Evidently, the face $R$ bounded by $C$} is contained in one  face of this \apriltwothree{drawing}.  \apriltwosix{As no face} of \apriltwothree{any simple drawing} of $K_4$ is incident with all of $e_1,e_2,e_3$, the result is immediate.  \end{proof}

\apriltwosix{A simple} consequence of Theorems \ref{th:connectedIntersection} and \ref{th:noThreeAtVertex} is that every face boundary of a simple drawing of $K_n$ consists of at most $n$ edge-segments.  

A {\em natural drawing\/} of $K_n$ is a simple drawing that has a Hamilton cycle \apriltwosix{$H$ bounding one face $F$}; it necessarily has the maximum number $\binom n4$ crossings.  For \apriltwosix{this drawing} it is obvious that:
\begin{enumerate}[topsep=0pt,itemsep=-5pt]\item \apriltwosix{$F$ is} incident with $n$ edge segments (in fact $n$ edges); and \item  every edge \apriltwosix{not in $H$} intersects the closure of $F$ in its two ends, as in Theorem \ref{th:connectedIntersection}. 
\end{enumerate}

\section{Proof of Theorem \ref{th:connectedIntersection}}\label{sec:connectedIntersection}

This section contains the proof of Theorem \ref{th:connectedIntersection}.

\begin{proof}\startClaims
We start by changing the drawing slightly so that no three edges are concurrent at a single crossing point.  By small isotopies, we can remove such concurrencies to yield a drawing in which any crossing involves precisely two edges.  Each face boundary of the original drawing has all its segments contained in a face boundary of the new drawing, so it suffices to prove the theorem for the new drawing.

We may assume that $e$  intersects  $C$.  If either there is no \apriltwosix{interval in $e\cap C$ consisting of more than a single point (such an interval is {\em non-trivial\/}) or every non-trivial} interval has both \apriltwosix{end vertices of $e$}, then we are done:  the former implies $e\cap C$ is at most the ends of $e$, \apriltwosix{while} the latter implies $e\cap C=e$. 

In the remaining case, there is a vertex $v$ incident with $e$ such that the first non-trivial interval $I$ \apriltwosix{in $e\cap C$} encountered upon traversing $e$ from $v$ is such that the end \apriltwosix{$x$ of $I$ that is furthest  from $v$ in $e$} is not an end of $e$.     If possible, choose $v$ so that $v\notin C$.  

\apriltwosix{The point $x$ is a crossing of $e$ with an edge $f$; these two edges} induce a $K_4$ with one crossing:  $x$.  The face of $K_n$ bounded by $C$ is contained in a face of the $K_4$ that is incident with the entire interval $I$.  It follows that $C\cap e\subseteq I\cup \{v\}$.  In particular, the other end $w$ of $e$ is not in $C$.

The choice of $v$ \apriltwosix{rather than $w$} implies either $v\in I$ or $v\notin C$.  In either case, $C\cap e=I$, as required.  \end{proof}

\section{Proof of Theorem \ref{th:regionsTriangles}}\label{sec:regionsTriangles}

This section contains the proof of Theorem \ref{th:regionsTriangles}.
\begin{proof}\startClaims
  It is clear that $R\subseteq \bigcap S^R_T$.  To show equality, we proceed by induction on $n$, the cases $n\le 4$ being easy and well-known.  \apriltwosix{Let $v$ be} any vertex of $K_n$.  Induction implies every \apriltwosix{face} of the drawing $D-v$ of $K_{n-1}$ is precisely the intersection of triangular sides.  
  
Let $R_{n-1}$ be the open face of $D-v$ containing $R$ and let $C_{n-1}$ be the simple closed curve bounding $R_{n-1}$.  The vertex $v$ is either in $R_{n-1}$ or not in $R_{n-1}\cup C_{n-1}$.

\begin{clm}\label{cl:atMostOneComponent}
If $e$ is incident with $v$, then $e\cap R_{n-1}$ has at most one component.  

Moreover, if $v$ is not in $R_{n-1}$ and $e\cap R_{n-1}$ has a component, then $e$ consists of the closed arc  from $v$ to the first intersection with $C_{n-1}$, the open arc $e\cap R_{n-1}$, and the closed arc from the other end of $e\cap R_{n-1}$ to the other end of $e$.  The first and last portions are disjoint from $R_{n-1}$. 
\end{clm} 

\begin{proof}
As we traverse $e$ from $v$ to its other end $w$, let $\alpha_1$ be the first component of $e\cap R_{n-1}$; this is an open arc with ends $x_1$ (nearer to $v$ in $e$) and $y_1$ in $C_{n-1}$.  (If $v\in R_{n-1}$, then $x_1=v$.)  Thus, $y_1$ is the first point of $C_{n-1}$ reached by our traversal from within $R_{n-1}$.  

If $y_1=w$, then we are done; in the remaining case, $y_1$ is a crossing of $e$ with an edge $f$ of $D-v$.  The triangle $T$ consisting of $f$ and $w$ has $R_{n-1}$ on the side containing $\alpha_1$.  

On the other hand, let $J$ be the $K_4$ induced by $e$ and $f$.  Then $D[J]$ has $e$ and $f$ crossing and shows that $\alpha_1$ and the portion $\alpha_2$ of $e$ from the crossing with $f$ to $w$ are on different sides of $T$.  Therefore, $\alpha_2$ cannot again cross into $R_{n-1}$.

The moreover assertion is evident from the preceding paragraphs.
\end{proof}

\ignore{ By way of contradiction, assume this is not the case. Let $i$ be minimal such that a drawing $D$ of $k_i$ has two facial regions defined by the same set of triangles, however, every drawing of $K_j$ has no two facial regions defined by the same intersection of triangles, for $j<i$. Clearly, $i>4$ as the theorem is true for every drawing of$K_3$ and $K_4$. By minimality, the two regions are contained in a unique region $\mathcal{C}$ in $K-i-v$, for a $v$ of ones choosing, such that $\mathcal{C}$ is uniquely defined by an intersection of triangles.
 \begin{clm}
Every edge incident to $v$ crosses $\mathcal{C}$ at most once from $Int(\mathcal{C})$ (the side with no vertices) to $Ext(\mathcal{C})$ (the side containing vertices).
 \end{clm}
\begin{proof}
 Suppose some edge incident to $v$ crosses $\mathcal{C}$ from  $Int(\mathcal{C})$ to $Ext(\mathcal{C})$. Consider the first such crossing with edge $f$ on the boundary of $\mathcal{C}$. The $K_4$ induced by $e$ and $f$ contains $\mathcal{C}$ in one facial region and the segment of $e$ after the $(e,f)$-crossing in another. Therefore, after the $(e,f)$-crossing, $e$ can not cross $\mathcal{C}$ again.
\end{proof}
}
If $v\in R_{n-1}$, then the star at $v$ partitions $R_{n-1}$ into $n-1$ sectors.  Each sector is precisely the intersection of the sides of triangles of $K_n$: use the sides of triangles of $D-v$ that determine $R_{n-1}$; and each triangle incident with $v$ has a side containing $R_n$.  The intersection of these sides is exactly $R_n$.

Henceforth, we assume $v\notin R_{n-1}$.  If no edge incident with $v$ has an arc in $R_{n-1}$, then $R_n=R_{n-1}$ and we are done.  Therefore we may assume the set $X$ of edges incident with $v$ and intersecting $R_{n-1}$ is not empty.  For each $e\in X$, let $u^e_1$ and $u^e_2$ \apriltwosix{be the ends $e\cap R_{n-1}$, labelled so that $u^e_1$ comes before $u^e_2$ as} we traverse $e$ from $v$.  Note that $v\notin C_{n-1}$, so the segment $s^e_1$ of $e$ from $v$ to $u^e_1$ is a non-trivial arc, as is the segment $s^e_2$ from $u^e_1$ to $u^e_2$.  Evidently $s^e_1$ is disjoint from $R_{n-1}$, while $s^e_2\setminus\{u^e_1,u^e_2\}= e\cap R_{n-1}$.  

\mayfive{We remark that, in principle, $u_1^e$ need not be the first intersection of $e$---traversed from $u$---with $C_{n-1}$.  However, this can only happen at a point at which three or more edges, including $e$, cross.  We can make a small adjustment of $e$ to obtain a new drawing in which a segment of $e$ near the crossing is also in $R_{n-1}$.  However, this new simple drawing violates Claim \ref{cl:atMostOneComponent}.  Thus, $u_1^e$ and (similarly) $u_2^e$ are the only intersections of $e$ with $C_{n-1}$.  This improvement was suggested by Kyn\v cl.
}


\begin{clm}\label{cl:twoCrossing}
If $e,f\in X$, then $u^e_1,u^e_2,u^f_2,u^f_1$ is the cyclic order of these four points in $C_{n-1}$ (in some orientation of $C_{n-1}$).
\end{clm}

\begin{proof}
If the vertices were interlaced on $C_{n-1}$, then the order would be  $u^e_1,u^f_2,u^e_2,u^f_1$.  Since $R_{n-1}\cup C_{n-1}$ is  a closed disc, $e\cap R_{n-1}$ and $f\cap R_{n-1}$ cross, contradicting the fact that $D$ is simple.

The remaining alternative is that the order is $u^e_1,u^e_2,u^f_1,u^f_2$, as in Figure \ref{fg:badOrder}.  In this case, the non-$v$ ends $w^e$ of $e$ and $w^f$ of $f$ are separated by a simple closed curve $\gamma$ that is composed of: an arc  in $R_{n-1}$ joining a point $x^e$ of $e\cap R_{n-1}$ to a point $x^f$ of $f\cap R_{n-1}$, but otherwise disjoint from $e\cup f$; the arc in $e$ from $v$ to $x^e$; and the arc in $f$ from $v$ to $x^f$.  

Since $D$ is simple and $R_{n-1}$ is a \apriltwosix{face} of $D-v$, the edge $w^ew^f$ does not cross $\gamma$ and so does not occur in $D-v$.  This contradicts the fact that $D-v$ is a simple drawing of $K_{n-1}$.
\end{proof}

\begin{figure}[!ht]
\begin{center}
\includegraphics[scale=.5]{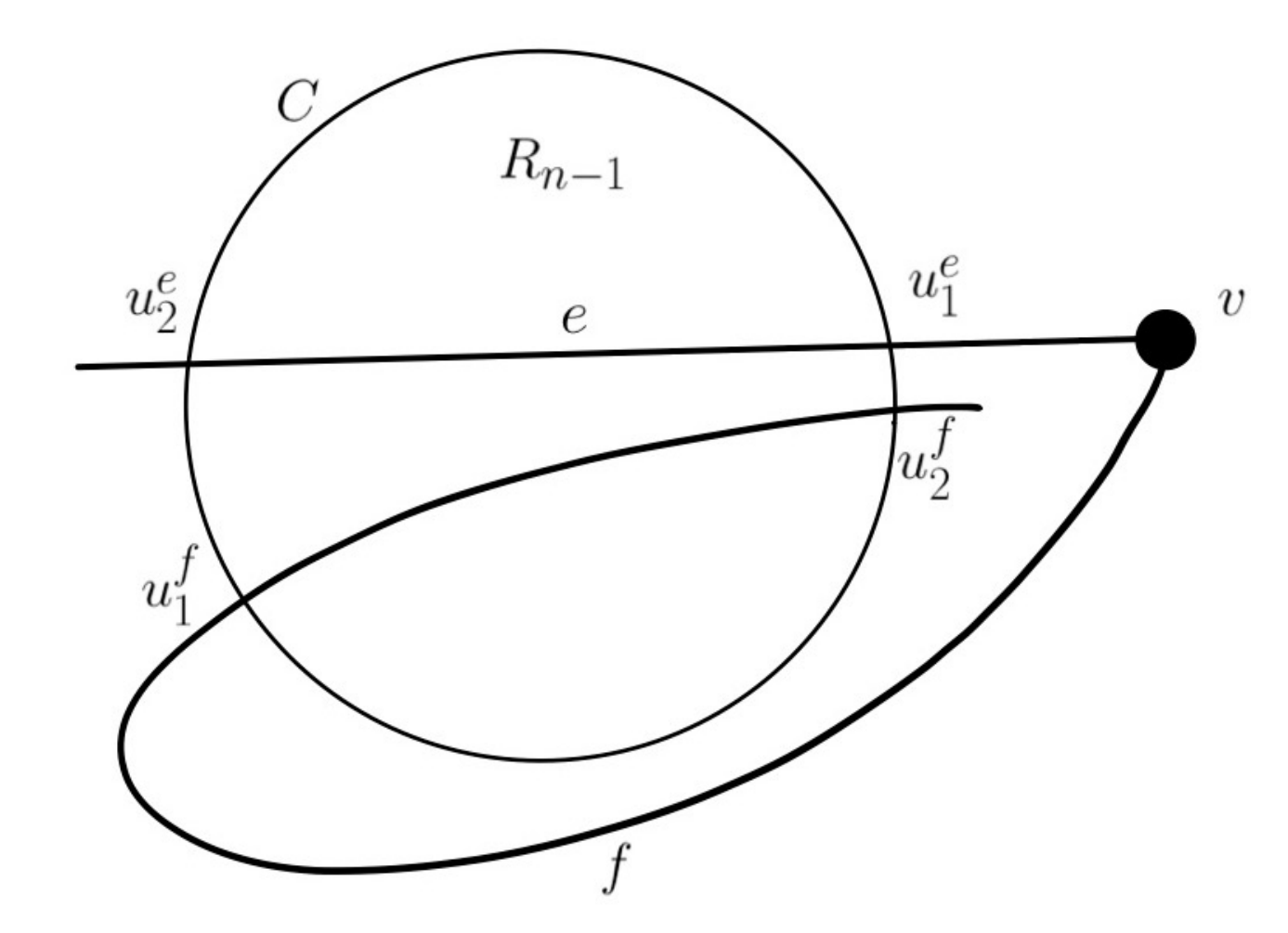}
\caption{The order $u^e_1,u^e_2,u^f_1,u^f_2$ does not occur.}\label{fg:badOrder}
\end{center}
\end{figure}

It follows from Claim \ref{cl:twoCrossing} that the edges of $X$ can be ordered $e_1,e_2,\dots,e_r$ so that the order of the intersections $u^{e_j}_j$ is $u^{e_1}_1,u^{e_2}_1,\dots,u^{e_r}_1,u^{e_r}_2,\dots,u^{e_2}_2,u^{e_1}_2$ and that these edges occur in this order in the cyclic sequence of edges incident with $v$ (although not necessarily consecutively).

For $i=2,3,\dots,r$, the triangle induced by $e_{i-1}$ and $e_i$ intersects $R_{n-1}$ in a \apriltwosix{face} $R^*_i$ of $D$ that is evidently the intersection of sides of the triangles in $K_n$.  It remains only to deal with the two remaining \apriltwosix{face}s, namely the \apriltwosix{face} $R^*_1$ incident only with $e_1$ and an arc in $C_{n-1}$ joining $u^{e_1}_1$ and $u^{e_1}_2$ and the analogous \apriltwosix{face} $R^*_{r+1}$ incident with $e_r$ and an arc in $C_{n-1}$.

\begin{clm}\label{cl:XnotAll}
$|X|<n-1$.   
\end{clm}

\begin{proof}  Suppose to the contrary that $|X|\ge n-1$.  Since there are only $n-1$ edges incident with $v$, $|X|=n-1$.  In the ordering $e_1,\dots,e_r$ of the edges of $X$ from the two paragraphs preceding the statement of this claim, $r=n-1$.  

Choose the traversal of $C_{n-1}$ to produce the cyclic order   $u^{e_1}_1,u^{e_2}_1$, $\dots,u^{e_{n-1}}_1,u^{e_{n-1}}_2$, $\dots,u^{e_2}_2,u^{e_1}_2$.  Let $f$ be the edge of $D-v$ that contains $u^{e_1}_1$.  The orientation of $C_{n-1}$ (from $u^{e_1}_1$ to $u^{e_2}_1$) induces an orientation of $f$ to match.  Let $x$ be a point of $f\cap C_{n-1}$ just after $u^{e_1}_1$ and let $w$ be the end of $f$ in that same direction.

\mayfive{There is an $i$ such that $vw=e_i$.  \maysix{Although we don't know precisely where $w$ is located, $f$ crosses into one side of  $\gamma$ at $u_1^{e_1}$, while $u^{e_i}_2$ is on the other side of $\gamma$.  Portions of $e_i$ and $f$ connect these through $w$.}  More precisely, let $\alpha$ be the arc that is the union of the  arc $e_i$ from $u^{e_i}_2$ to $w$ and the arc in $f$ from $w$ to $x$.  Let $\gamma$ be the simple closed curve that is the union of the portions of $e_1$ and $e_{i}$ from $v$ to $u^{e_1}_1$ and $u^{e_{i}}_1$, respectively, and an arc in $R_{n-1}$ joining $u^{e_1}_1$ and $u^{e_{i}}_1$.  See Figure \ref{fg:notSimple}.}

\begin{figure}[!ht]
\begin{center}
\includegraphics[scale=.6]{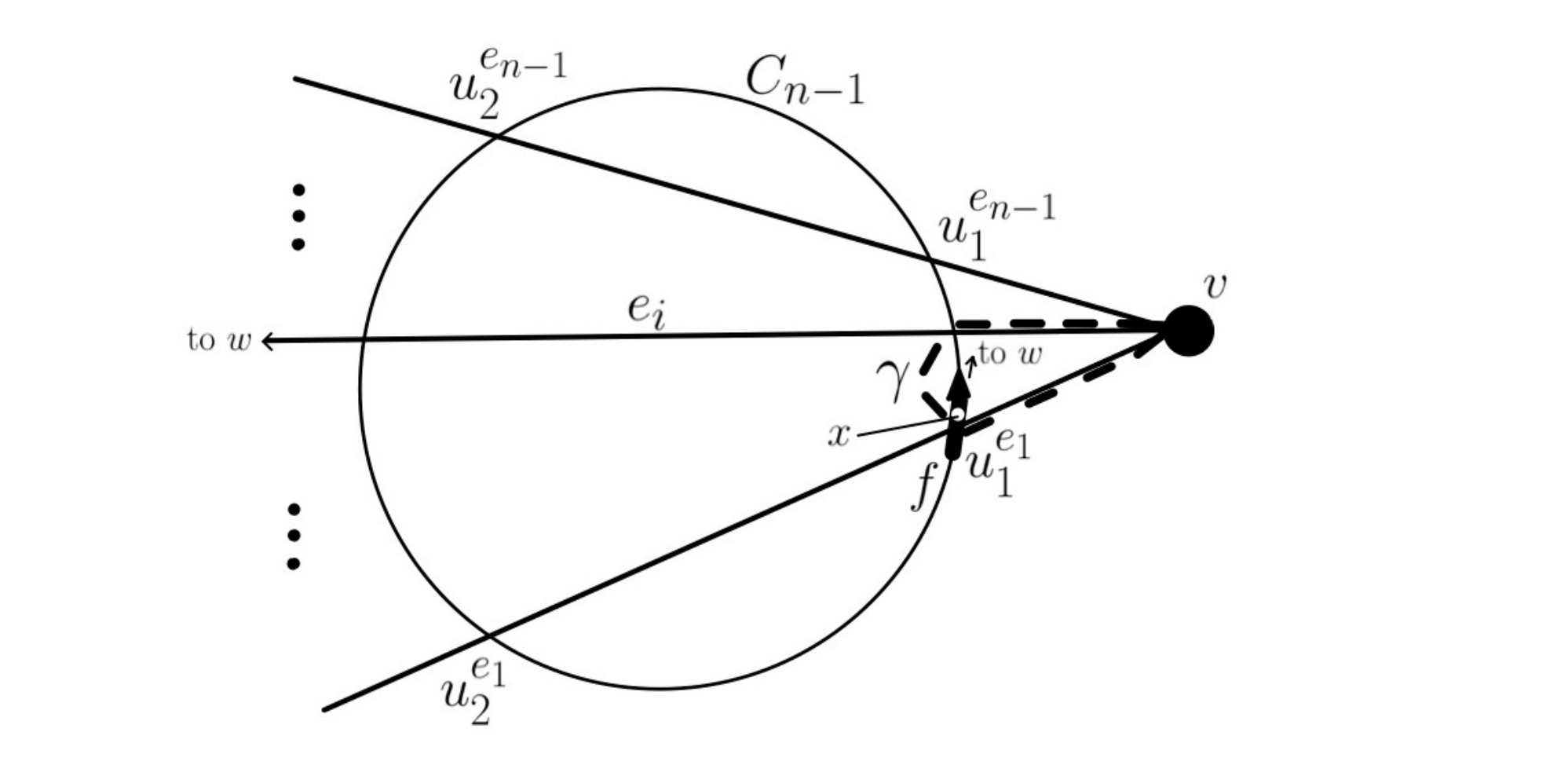}
\caption{The location of $e_i=vw$, $f$, and $\gamma$.}\label{fg:notSimple}
\end{center}
\end{figure}

\mayfive{Evidently, $\gamma$ separates $x$ from $u^{e_i}_2$, so $\alpha$ crosses $\gamma$.  As $\alpha$ is disjoint from $R_{n-1}$, $\alpha$ intersects either $e_1$ or $e_{i}$.  Because the drawing is simple, it is not the arc in $\alpha$ contained in $vw$ that intersects either $e_1$ or $e_{i}$.} 

\mayfive{Because $D$ is simple, $f$ cannot cross $e_1$ at a point other \maysix{than $u^{e_1}_1$}.  Since $f$ and $e_i$ have the common end $w$, the portion of $\alpha$ contained in $f$ also does not intersect either $e_1$ or $e_i$, a contradiction.} \end{proof}

The proof is completed by letting $w$ be a vertex such that $vw$ does not cross $R_{n-1}$.  The triangle induced by $e_1$ and $w$ has $R^*_1$ on one side and the rest of $R_{n-1}$ on the other.  Therefore, $R^*_1$ (and symmetrically $R^*_{r+1}$) are precisely the intersections of sides of triangles of $K_n$.
\end{proof}

\end{document}